\documentclass[11pt]{amsart}

\usepackage{amsmath,amsthm,amsfonts,amssymb,latexsym,mathrsfs,color,cite,tikz,cases,url,enumerate}
\usepackage{color}
\usepackage[pagebackref]{hyperref}
\usepackage[capitalize]{cleveref}
\usepackage{enumitem}
\usepackage{verbatim}

\allowdisplaybreaks

  \makeatletter
  \newcommand{\numberlike}[2]{%
     \expandafter\def\csname c@#1\endcsname{%
         \expandafter\csname c@#2\endcsname}%
  }
  \makeatother

  \def\DefaultNumberTheoremWithin{section}

  \theoremstyle{plain}
  \newtheorem{lemma}{Lemma}
     \numberwithin{lemma}{\DefaultNumberTheoremWithin}
     \labelformat{lemma}{Lemma~#1}
  \newtheorem{theorem}{Theorem}
     \numberwithin{theorem}{\DefaultNumberTheoremWithin}
     \numberlike{theorem}{lemma}
     \labelformat{theorem}{Theorem~#1}
  \newtheorem*{maintheorem*}{Corollary of Main Theorem}
  \newtheorem{corollary}{Corollary}
     \numberwithin{corollary}{\DefaultNumberTheoremWithin}
     \numberlike{corollary}{lemma}
     \labelformat{corollary}{Corollary~#1}
  \newtheorem{proposition}{Proposition}
     \numberwithin{proposition}{\DefaultNumberTheoremWithin}
     \numberlike{proposition}{lemma}
     \labelformat{proposition}{Proposition~#1}
  \newtheorem{conjecture}{Conjecture}
     \numberwithin{conjecture}{\DefaultNumberTheoremWithin}
     \numberlike{conjecture}{lemma}
     \labelformat{conjecture}{Conjecture~#1}

  \theoremstyle{definition}
  
     \numberwithin{definition}{\DefaultNumberTheoremWithin}
     \numberlike{definition}{lemma}
     \labelformat{definition}{Definition~#1}
  \newtheorem{question}{Question}
     \numberwithin{question}{\DefaultNumberTheoremWithin}
     \numberlike{question}{lemma}
     \labelformat{question}{Question~#1}
  
     \numberwithin{problem}{\DefaultNumberTheoremWithin}
     \numberlike{problem}{lemma}
     \labelformat{problem}{Problem~#1}
    
  \theoremstyle{remark}
  
     \numberwithin{remark}{\DefaultNumberTheoremWithin}
     \numberlike{remark}{lemma}
     \labelformat{remark}{Remark~#1}
  \newtheorem{example}{Example}
     \numberwithin{example}{\DefaultNumberTheoremWithin}
     \numberlike{example}{lemma}
     \labelformat{example}{Example~#1}
     
 \newtheorem{claim}{Claim}
     \numberwithin{claim}{\DefaultNumberTheoremWithin}
     \numberlike{claim}{lemma}
     \labelformat{claim}{Claim~#1}
\labelformat{figure}{Figure~#1}
\labelformat{chapter}{Chapter~#1}
\labelformat{appendix}{Appendix~#1}
\labelformat{section}{Section~#1}
\labelformat{subsection}{Subsection~#1}

\def\RR{\mathbb{R}}
\def\NN{\mathbb{N}}

\def\F{\mathcal{F}}

\newcommand\cgreen[1]{\color{magenta} #1 \color{black}}
\def\des{\mathrm{des}}

\title{Total positivity and two inequalities by
  Athanasiadis and Tzanaki
}

\author{Lili Mu} 
\address{School of Mathematics and Statistics, Jiangsu Normal University, Xuzhou 221116, PR China} 
\email{lilimu@jsnu.edu.cn}

\author{Volkmar Welker}
\address{Philipps-Universit\"at Marburg, Fachbereich Mathematik und Informatik, 35032 Marburg, Germany}
\email{welker@mathematik.uni-marburg.de}
\begin{document}

\begin{abstract} 
Let $\Delta$ be a $(d-1)$-dimensional simplicial complex and
$h^ \Delta = (h_0^ \Delta ,\ldots, h_d^ \Delta)$ its $h$-vector.
For a face uniform subdivision operation $\F$ we write $\Delta_\F$ for
the subdivided complex and $H_\F$ for the matrix such that
$h^ {\Delta_\F} = H_\F h^ \Delta$.

In connection with the real rootedness of symmetric decompositions
Athanasiadis and Tzanaki studied for strictly positive $h$-vectors the inequalities $\frac{h_0^ \Delta}{h_1^ \Delta} \leq \frac{h_1^\Delta}{h_{d-1}^ \Delta} \leq \cdots \leq 
\frac{h_d^ \Delta}{h_0^\Delta}$ and $\frac{h_1^\Delta}{h_{d-1}^\Delta} \geq \cdots
    \geq \frac{h_{d-2}^\Delta}{h_2^\Delta} \geq \frac{h_{d-1}^\Delta}{h_1^\Delta}$. 
In this paper we show that if the inequalities holds for a simplicial complex $\Delta$ and $H_\F$ is TP$_2$ (all entries and two minors are non-negative) then 
the inequalities hold for $\Delta_\F$. 

We prove that if $\F$ is the barycentric subdivision then $H_\F$ is
TP$_2$. If $\F$ is the $r$\textsuperscript{th}-edgewise subdivision then
work of Diaconis and Fulman shows $H_\F$ is TP$_2$. Indeed in this case
by work of Mao and Wang $H_\F$ is even TP.
\end{abstract}

\maketitle
\vspace{-0.5in}

\section{Introduction}

We study when a subdivision operation of simplicial complexes preserves a set of inequalities on the $h$-vector of the complex, first considered by Athanasiadis and Tzanaki in \cite{AT21} in the context of 
real rooted symmetric decompositions of $h$-polynomials. 

For a (finite) simplicial complex $\Delta$ of dimension
$d-1$ the $h$-vector $h^\Delta = (h_0^\Delta,\ldots, h_d^\Delta)$
is an encoding of the face numbers
of the simplicial complex (see for example \cite{BB} for definitions and
background). Athanasiadis and Tzanaki \cite{AT21}
study the following two conditions and inequalities

\begin{align}
\label{eq:athtza}
h_0^ \Delta, \ldots,h_d^ \Delta > 0 \text{ and }
    \frac{h_0^\Delta}{h_d^\Delta} \leq & \frac{h_1^\Delta}{h_{d-1}^\Delta} \leq \cdots
    \leq \frac{h_{d-1}^\Delta}{h_1^\Delta} \leq \frac{h_d^\Delta}{h_0^\Delta},
\end{align}

\begin{align}
\label{eq:athtza2}
  h_1^ \Delta, \ldots,h_{d-1}^ \Delta > 0\,,\, h_d^ \Delta=0 \text{ and } \frac{h_1^\Delta}{h_{d-1}^\Delta} \geq \cdots
    \geq \frac{h_{d-2}^\Delta}{h_2^\Delta} \geq \frac{h_{d-1}^\Delta}{h_1^\Delta}.
\end{align}

By $h_0^\Delta = 1$ and the Dehn-Summerville equations 
$h_i^\Delta = h_{d-i}^\Delta$ it follows that \eqref{eq:athtza} and \eqref{eq:athtza2} 
(except for $h_d^\Delta = 0$) hold
with equality for any Gorenstein\textsuperscript{*} complex.
In \cite[Question 7.2]{AT21} the authors ask if \eqref{eq:athtza} 
holds for all $2$-Cohen-Macaulay simplicial complexes and in
\cite[Question 7.3]{AT21} they ask which triangulations of balls satisfy 
\eqref{eq:athtza2}. 
Since $2$-Cohen-Macaulayness  
is a topological property a positive answer to \cite[Question 7.2]{AT21}
would imply that the inequalities should be preserved under
any subdivision of a $2$-Cohen-Macaulay complexes. We do not know how to
leverage the fact that the complex is $2$-Cohen-Macaulay when trying to 
show that \eqref{eq:athtza} is preserved under subdivisions. Instead we show
in \ref{pro:preserve} that if we subdivide by a face uniform subdivision 
$\F$ (see \cite{Atha22}) then the
preservation of \eqref{eq:athtza} or \eqref{eq:athtza2} follows from the TP$_2$ property of the
$h$-vector transformation matrix $H_\F$ of the subdivision. Recall that a real matrix is totally positive or TP if all its minors are non-negative, it 
has the
TP$_r$ property if all its $(k \times k)$-minors 
are non-negative for $k=1,\ldots, r$. 

TP-theory arises from unimodality and log-concavity questions in a quite 
natural way. Since Brenti \cite{Bren89} applied
this theory to prove and generate unimodal, log-concave sequences, this theory 
has proven to be a very useful tool in combinatorics and has been applied
frequently and in a wide range of contexts 
(see e.g., \cite{MMW22,Muk24,PSZ23,Skan01,Zhu23}).

We show for two prominent face uniform subdivisions that they satisfy the assumption of \ref{pro:preserve}.  

In \ref{tp} we prove that $H_\F$ is TP$_2$ if $\F$ is the barycentric
subdivision. The latter is defined via a refined descent statistics on the symmetric group.
As a byproduct of our studies we obtain in \ref{thm:inverse} 
additional structural insight into the combinatorics of $H_\F^ {-1}$ in this case. 

We recall results by Diaconis \& Fulman \cite{DF09} and Mao \& Wang \cite{MW22} which imply the $H_\F$ is TP$_2$ \cite{DF09} and even TP \cite{MW22} if $\F$ is the $r$\textsuperscript{th}-edgewise subdivision. 

In particular, \eqref{eq:athtza} or \eqref{eq:athtza2} is preserved by barycentric (see \ref{cor:bary}) 
and $r$\textsuperscript{th}-edgewise subdivision (see \ref{cor:edge}).

We also provide examples which show that there are face uniform 
subdivisions $\F$ for which $H_\F$ is not TP$_2$. Nevertheless,
we were not able to produce an example which shows that such a 
subdivision can destroy the validity of \eqref{eq:athtza}.

\section{Triangulations, \texorpdfstring{$f$}~- and \texorpdfstring{$h$}~-vectors}

For a finite set $\Omega$ a simplicial complex over ground set $\Omega$ is
a set of subsets of $\Omega$ such that $\sigma \subseteq \tau \in \Delta$
implies $\sigma \in \Delta$. We call a $\sigma \in \Delta$ a face of $\Delta$
and the dimension of a face $\sigma$ is $\dim(\sigma) = \# \sigma -1$. 
The dimension of $\Delta$ is $\max_{\sigma \in \Delta} \dim(\sigma)$. 
If $\Delta$ is a $(d-1)$-dimensional simplicial complex then the
vector $f^\Delta = (f_{-1}^ \Delta,\ldots, f_{d-1}^ \Delta)$ where
$f_i^ \Delta = \# \big\{\, \sigma \in \Delta~\big|~\dim(\sigma) = i\,\big\}$
is called the $f$-vector of $\Delta$
and $f^ \Delta(x) = \sum_{i=0}^ d f_{i-1}^ {\Delta} x^ {d-i}$ is called the $f$-polynomial of
$\Delta$. Expanding $f^ \Delta(x-1) = h^ \Delta(x) = \sum_{i=0}^ d h_i^ \Delta x^ {d-i}$ yields the $h$-polynomial
with coefficient sequence $h^\Delta = (h_0^ \Delta,\ldots, h_d^ \Delta)$ the $h$-vector of $\Delta$.
To each simplicial complex $\Delta$ there is a geometric realization $|\Delta|$ in some
real vectorspace in which
each face $\sigma$ of $\Delta$ is represented by a geometric simplex $|\sigma|$ of dimension $\dim(\sigma)$
such that $|\sigma| \cap |\tau| = |\sigma \cap \tau|$ for all $\sigma,\tau \in \Delta$. 
A simplicial complex $\Delta'$ is called a face uniform triangulation or subdivision of 
$\Delta$ if there are geometric realizations $|\Delta| = |\Delta'|$ such that
\begin{itemize}
  \item each $|\sigma|$ for $\sigma \in \Delta$ is a union of $|\sigma'|$ for 
$\sigma' \in \Delta'$ and 
\item there are numbers $f_{ij}$, $0 \leq i \leq j \leq \dim (\Delta)$ such that
for any $\sigma \in \Delta$ we have $f_{ij} = \# \{ \tau \in \Delta'~:~ |\tau| \subseteq |\sigma|, \dim (\tau) = i\}$.
\end{itemize}

We write $\F$ for the triangular array $(f_{ij})_{0 \leq i,j}$. Since we are only interested in the 
enumerative aspects of the triangulation we 
write $\Delta_\F$ for $\Delta'$ in this case and speak of $\F$ as a face uniform triangulation (in dimension $d-1$).

The following summarizes the results from \cite{Atha22},
which we will use in this paper.

\begin{proposition}[Theorem 1.1, Proposition 4.6 \cite{Atha22}]
    Let $\F$ be a face uniform triangulation in
    dimension $d-1$. 
    Then there is a matrix $H_\F = (h_{ij})_{0 \leq i,j \leq d}$ such that for any simplicial complex
    $\Delta$ of dimension $d-1$ we have 

    $$h^{\Delta_\F} = H_\F \, h^\Delta.$$
    Moreover, we have $h_{ij} = h_{d-i,d-j}$
    for $0 \leq i,j \leq d$. 
\end{proposition}    

Next we formulate and prove the result which allows us to approach the preservation of \eqref{eq:athtza} or \eqref{eq:athtza2} under face uniform
triangulations.

\begin{proposition}
\label{pro:preserve}
    Let $\F$ be a face uniform subdivision such that
    $H_\F$ is TP$_2$. 
    
    Then for every $(d-1)$-dimensional simplicial complex 
    $\Delta$ 
    \begin{itemize}
        \item[(i)] 
    satisfying \eqref{eq:athtza}  
    we have that
    $\Delta_\F$ satisfies \eqref{eq:athtza},
    \item[(ii)] satisfying \eqref{eq:athtza2} 
    we have that
    $\Delta_\F$ satisfies \eqref{eq:athtza2}.
    \end{itemize}
\end{proposition}

 \begin{proof}
    First consider (i) and assume $\Delta$ satisfies \eqref{eq:athtza}.
    Consider the matrix $A \in \RR^{(d+1) \times 2}$ with
    first column vector $(h_d^ \Delta,\ldots, h_0^ \Delta)^t$  and
    second column vector $(h_0^ \Delta,\ldots, h_d^ \Delta)^t$. 
    From \eqref{eq:athtza} and the positivity of the $h$-vector we deduce that 
    the $(2 \times 2)$-minors 
    $h_{i+1}^ \Delta \cdot h_{d-i}^ \Delta-h_i^ \Delta \cdot h_{d-i-1}^\Delta $ of consecutive 
    rows from $A$ are non-negative. It follows by \ref{lem:consecutive} that 
    $A$ is TP$_2$. Using \ref{lem；product}, we deduce from $H_\F$ being TP$_2$ and
    $A$ being TP$_2$ that $H_\F \cdot A$ is
    TP$_2$.
    Set $H_\F \cdot A = \big(\,b_{ij}\,\big)_{\genfrac{}{}{0pt}{}{i=0,\ldots,d}{j=1,2}}$.
    Then 
    \begin{align*}
        b_{i1} & = h_{i0} h_d^ \Delta + \cdots + h_{id}h_0^ \Delta \\
               & = h_{d-i,d} h_d^ \Delta + \cdots + h_{d-i,0} h_0^ \Delta \\
               & = h_{d-i}^{\Delta_\F}.
    \end{align*}
    and
    \begin{align*}
        b_{i2} & = h_{i0} h_0^ \Delta + \cdots + h_{id}h_d^ \Delta =
    h_i^{\Delta_\F}.
    \end{align*}
    Since $H_\F \cdot A$ is TP$_2$ it follows that
    $$b_{i1}\cdot b_{i+1,2} -b_{i2} \cdot b_{i+1,1}
    = h^{\Delta_\F}_{i+1} \cdot h^{\Delta_\F}_{d-i} 
    -h^{\Delta_\F}_i \cdot h^{\Delta_\F}_{d-i-1}
    \geq 0.$$
    This then implies \eqref{eq:athtza} for $\Delta_\F$.   

    For (ii) we use a very similar argumentation. Assume $\Delta$ satisfies \eqref{eq:athtza2}. Consider the matrix $A \in \RR^{(d+1) \times 2}$ with
    first column vector $(h_0^ \Delta,\ldots, h_d^ \Delta)^t$  and
    second column vector $(h_d^ \Delta,\ldots, h_0^ \Delta)^t$.
    From \eqref{eq:athtza2} and the positivity of the $h$-vector we deduce that 
    the $(2 \times 2)$-minors 
    $h_{i}^ \Delta \cdot h_{d-i-1}^ \Delta-h_{i+1}^ \Delta \cdot h_{d-i}^\Delta $ of consecutive 
    rows from $A$ are non-negative for $i=1,\ldots, d-2$.
    For $i=0$ the minor is $h_{0}^ \Delta \cdot h_{d-1}^ \Delta-h_1^ \Delta \cdot h_{d}^\Delta = h_{d-1}^ \Delta > 0$ and
    for $i=d-1$ the minor is
    $h_{d-1}^ \Delta \cdot h_{0}^ \Delta-h_{d}^ \Delta \cdot h_{1}^\Delta 
    = h_{d-1}^ \Delta > 0$. Using \ref{lem:consecutive}, it follows that 
    $A$ is TP$_2$. By \ref{lem；product}, we deduce from $H_\F$ being TP$_2$ and
    $A$ being TP$_2$ that $H_\F \cdot A$ is
    TP$_2$. The rest of the argument is analogous to case (i)
    taking into account that the roles of the two columns of $A$ are
    reversed.
 \end{proof}

We will apply this result to the case when $\F$ is the barycentric and
the case when $\F$ is the $r$\textsuperscript{th}-edgewise subdivision.
For the definition of barycentric and $r$\textsuperscript{th} edgewise
subdivision we refer the reader to \cite{MWe17}.

Our main contribution is the following result on barycentric subdivision
which is proved in \ref{sec:bary}.

\begin{theorem}\label{tp}
  Let $\F$ be the barycentric subdivision. Then 
  $H_{\F}$ is TP$_2$.
\end{theorem}

As an immediate corollary of 
\ref{pro:preserve} and \ref{tp} we obtain:

\begin{corollary} \label{cor:bary}
    Let $\F$ be the barycentric subdivision. If
    $\Delta$ satisfies \eqref{eq:athtza} or satisfies \eqref{eq:athtza2}, then so
    does $\Delta_\F$. 
\end{corollary}

In case $\F$ is the $r$\textsuperscript{th}-edgewise 
subdivision an even stronger result holds.

\begin{theorem}[Mao, Wang \cite{MW22}] \label{thm:edge}
  Let $\F$ be the $r$\textsuperscript{th}-edgewise
subdivision. Then $H_\F$ is TP.
\end{theorem} 

Note that the fact that
$H_\F$ is TP$_2$ was proved before by Diaconis and Fulman \cite{DF09}. 
Again as an immediate corollary we get:

\begin{corollary} \label{cor:edge}
    Let $\F$ be the $r$\textsuperscript{th}-edgewise subdivision. If
    $\Delta$ satisfies \eqref{eq:athtza} or satisfies \eqref{eq:athtza2}, then so
    does $\Delta_\F$. 
\end{corollary}

Let us now look into generalizations of these results. 
Based on experimental evidence we conjecture that the conclusion of
\ref{thm:edge} also holds for barycentric subdivision.

\begin{conjecture}
  Let $\F$ be the barycentric subdivision.
 Then $H_\F$ is TP.
\end{conjecture} 

On the other hand, it is easy to see that not all face uniform subdivisions $\F$ have a TP$_2$ transformation matrix $H_\F$. 

\begin{example} \label{ex:notp}
  Let $\F$ be the subdivision of $d$-dimensional simplicial complexes
  which replaces each $d$-simplex by a cone over its boundary. 
  The $f_{ij}$ here take following form 
  $$f_{ij} = \left\{ \begin{array}{ccc} 0 & \text{ for } 0 \leq j < i < d \\
                                        1 & \text{ for } j=i < d \\
                                        \binom{d+1}{i} & \text{ for } 0 \leq j < d = i \end{array} \right.$$
  Then $H_\F$ takes the following form:
  $$H_\F = \begin{pmatrix} 1 & 0 & 0 & 0 & \cdots & 0 & 0 \\
                           1 & 2 & \cgreen{1} & \cgreen{1} & \cdots & 1 & 1 \\
                           1 & 1 & \cgreen{2} & \cgreen{1} & \cdots & 1 & 1 \\
                           \vdots & \vdots & \vdots & \vdots & \vdots & \vdots & \vdots \\
                           1 & 1 & 1 & 1& \cdots & 2 & 1 \\
                           0 & 0 & 0 & 0& \cdots & 0 &  1
                           \end{pmatrix}.$$
                           For example the highlighted $(2 \times 2)$-submatrix has negative determinant.
                           If $\F^ n$ is the $n$\textsuperscript{th} iteration of this 
                           subdivision then $H_{\F^ n} = H_\F^ n$. It is
                           easily checked that for high enough $n$ 
                           in those powers 
                           there will be $(2 \times 2)$-submatrices 
                           with arbitrarily negative determinant. 
                           Nevertheless, we were not able to find 
                           a simplicial complex satisfying 
                           \eqref{eq:athtza} for which one of those
                           iterations breaks these inequalities.
\end{example}   

The subdivision from \ref{ex:notp} is quite special as the subdivision spares faces of certain dimensions from being subdivided. We could not find a 
face uniform triangulation $\F$ which subdivides faces of all dimensions 
for which $H_\F$ is not TP$_2$. 

\begin{question}
    Which geometric conditions on a face uniform triangulation $\F$ imply that $H_\F$ is
    TP$_2$ (resp., TP) ?   
\end{question}

For example, experiments show that in small dimensions
the antiprism-triangulations (see
\cite{ABK22}) and the interval-subdivision (see \cite{AN20}) 
have TP$_2$ and even TP matrices $H_\F$. Given this experimental evidence 
and the results from this paper we strongly believe that it is possible
to define a class of face uniform triangulations whose matrices $H_\F$ are
TP which includes all those as special cases.

\section{TP\texorpdfstring{$_2$}~~for barycentric subdivision}
\label{sec:bary}

In this section we prove \ref{tp}: 
the TP$_2$-property for $H_\F$ where $\F$ is
the barycentric subdivision. This will turn out to be quite involved.

Let $$D(\sigma)=\big\{\,i\in [d-1]\,\big|\,\sigma(i)>\sigma(i+1)\,\big\}$$
be the descent set of the permutation $\sigma$,
$\des(\sigma):=\sharp D(\sigma)$ be its number of descents and
$S_d$ be the symmetric group on $[d]$.
For $0\leq i,j\leq d-1$, we denote by $A(d,i,j)$ the number of permutations $\sigma\in S_d$ such that
$\des(\sigma)=i$ and $\sigma(d)=d-j$.
We define $A(d,i,j)$ for all $d\geq 1$ and all integers $i$ and $j$.
In particular $A(d,i,j)=0$ if $i\leq -1$ or $i\geq d$.
With these conventions it is easily seen that 
\begin{align} \label{eq:posentry}
   A(d,i,j) = 0 & \Leftrightarrow \begin{array}{c} i=0  \text{ and } j \neq 0  \\   \text{or} \\ i=d-1  \text{ and } j \neq d-1.
   \end{array} 
\end{align}

The following connects $A(d,i,j)$ to the barycentric subdivision.

\begin{lemma}[Theorem 1 \cite{BW08}] \label{baryh}
  Let $\F$ be the barycentric subdivision of $(d-2)$-dimensional 
  simplicial complexes. 
  Then $H_\F = \big(\, A(d,i,j)\,\big)_{0 \leq i ,j \leq d-1}$.
  \end{lemma}

We will employ a bijection between labeled paths and permutations given in 
\cite{BE00,Bona22}. Let $P(d)$ be the set of $d$-tuples $\big(\,(a_1,u_1), \ldots ,(a_d,u_d)\,\big)$
in $\big(\,\{ E,N\} \times \NN \,\big)^d$, 
satisfying: 
\begin{itemize}
  \item[(L1)]  $a_1 = E$ and $u_1=1$,
  \item[(L2)] if $a_i = a_{i+1} = N$, or $a_i=a_{i+1} = E$ then 
	  $u_i\geq u_{i+1}$,
  \item[(L3)] if $a_i \neq a_{i+1}$ then $u_i+ u_{i+1}\leq i+1$.
\end{itemize}

Interpreting $N$ as a step north and $E$ as a step east, we consider an element of $P(d)$ as 
a northeast path of length $d$ with each step $(a_i,u_i)$ labeled with a natural number $u_i$. 

The the bijection $\Psi : S_d \rightarrow P(d)$ is defined as follows.
For $\sigma = \sigma_1\cdots \sigma_d \in S_d$ we set $\Psi(\sigma) = 
\big(\,(a_1,u_1), \ldots , (a_d,u_d)\,\big)$ where:

\begin{itemize}
    \item $(a_1,u_1) = (E,1)$
    \item for $2 \leq i \leq d$ we obtain $(a_i,u_i)$ as follows.
    Let $\tau = \tau_1\cdots\tau_i \in S_i$ be the
permutation such that for $1 \leq \ell < j \leq i$ we have $$\tau_\ell < \tau_j \Leftrightarrow
\sigma_\ell < \sigma_j.$$

\begin{itemize}
  \item If the position $i-1$ in $\sigma$ or equivalently $\tau_{i-1}$ is a descent, let the $a_i = N$ and set $u_i=\tau_i$.
  \item If the position $i-1$ in $\sigma$ or equivalently $\tau_{i-1}$ is an ascent, let the $a_i= E$ and set $u_i=i+1-\tau_i$.
\end{itemize}
\end{itemize}

The example from \cite[FIG.2]{BE00} shown in \ref{squ11} illustrates the definition of $\Psi$.

\begin{figure}[ht]
\centering
\begin{tikzpicture}
\draw (0,0)--(2,0)--(2,2)--(3,2)--(3,3);

\foreach \x in {0,1,2,}
\foreach \y in {0}
\fill (\x,\y) circle(0.05); 
\foreach \x in {2}
\foreach \y in {1,2}
\fill (\x,\y) circle(0.05); 
\fill (3,2) circle(0.05); \fill (3,3) circle(0.05);

\foreach \x in {0.5,1.5}
\foreach \y in {0}
\node at(\x,\y)[above,font=\scriptsize] {$1$};
\node at(2,0.5)[right,font=\scriptsize] {$2$};
\node at(2,1.5)[right,font=\scriptsize] {$1$};
\node at(2.5,2)[above,font=\scriptsize] {$1$};
\node at(3,2.5)[right,font=\scriptsize] {$5$};
\end{tikzpicture}
\caption{The image of the permutation $243165$.}
\label{squ11}
\end{figure}

\begin{theorem}[\cite{BE00}]
  The map $\Psi : S_d \rightarrow P(d)$ is a bijection.
\end{theorem}

By construction for $\sigma = \sigma_1 \cdots \sigma_d$ and $\Psi(\sigma) = 
\big(\, (a_1,u_1),\ldots, (a_d,u_d)\,\big)$ we have:

\begin{align}
    \label{des}
\sigma \text{ has } i \text{ descents } & \Leftrightarrow \Psi(\sigma) \text{ has } i \text{ steps } 'N'
\end{align} 
and 
\begin{align}
    \label{last} \sigma_d = d-j & \Leftrightarrow u_d = \left\{ \begin{array}{ccc} d-j & \text{ if } & a_d=N \\ j+1 & \text{ if } & a_d=E \end{array} \right. .
\end{align}

Let $P(d,i,j)$ be the set of labeled paths in $P(d)$ with $i$ steps $N$ and $$u_d = \left\{ \begin{array}{ccc} d-j & \text{ if } & a_d=N \\ j+1 & \text{ if } & a_d=E \end{array} \right..$$
It follows that:

\begin{corollary}
 $\Psi$ restricts to a bijection between $A(d,i,j) \subseteq S_d$  and $P(d,i,j)
  \subseteq P(d)$ for all $0\leq i,j\leq d-1$.
\end{corollary}

For two labeled paths $P=\big(\,(a_1,u_1),\ldots, (a_d,u_d)\,\big)$
and $Q = \big(\,(b_1,v_1),\ldots, (b_d,v_d)\,\big)$ in $P(d)$ and for
$1 \leq r < d$ consider the following conditions:

\begin{itemize}
    \item[(A$_r$)]  
    $a_r=a_{r+1} = N$ and $b_r=b_{r+1} = N$
    and 
    both $u_r\geq v_{r+1}$ and $v_r\geq u_{r+1}$ hold.
    \item[(B$_r$)] 
     $a_r=a_{r+1} = E$ and $b_r=b_{r+1} = E$
    and 
    both $u_r\geq v_{r+1}$ and $v_r\geq u_{r+1}$ hold.
    \item[(C$_r$)] either steps 
    $(N,u_r), (E,u_{r+1})$ and $(N,v_r)$, $(E,v_{r+1})$
    or steps
    $(E,u_r), (N,u_{r+1})$ and $(E,v_r)$, $(N,v_{r+1})$
    and both $u_r+v_{r+1}\leq r+1$ and $v_r+u_{r+1}\leq r+1$ hold.
\end{itemize}

\begin{lemma}
    \label{lem:parallel}
    Let $P = \big(\,(a_1,u_1),\ldots ,(a_d,u_d)\,\big)$ and
    $Q = \big(\,(b_1,v_1),\ldots, (b_d,v_d)\,\big)$ be two labeled paths in $P(d)$. 
    Let $1 \leq s < t \leq d$ such that
    we have $a_r = b_r$ for 
    $s \leq r \leq t$. Assume that either
    \begin{itemize}
        \item $a_s = b_s = E$ and $u_s < v_s$  or 
        \item $a_s = b_s = N$ and $u_s > v_s$ 
    \end{itemize}
    holds
  and that there is no $s \leq r <t$ such that one of 
    (A$_r$), (B$_r$) or (C$_r$) holds.
    
    Then for $s \leq r \leq t$ we have
    $u_r < v_r$ if $a_r=b_r = E$ and $u_r > v_r$ if $a_r=b_r = N$.
\end{lemma}
\begin{proof}
   We proceed by induction on $r$.
   By assumption the claim holds for $r = s$. 

   Let $s < r < t$ and assume the claim is proved for $r-1$.
   
   \begin{itemize}
   \item $a_{r-1}=b_{r-1}=a_r= b_r = N$

    By induction we know that 
    $u_{r-1} > v_{r-1}$.
    Then $v_r < v_{r-1} < u_{r-1}$. Since (A$_{r-1}$)
    does not hold, one of $u_{r-1}\geq v_{r}$ and $v_{r-1}\geq u_{r}$
    must be false. It follows that $v_{r-1} \geq u_r$ is false and
    hence $u_r > v_{r-1} > v_r$.
   \item $a_{r-1}=b_{r-1}=a_r= b_r = E$ 

   The proof of this case is analogous to the preceding one
   with (A$_r$) is replaced by (B$_r$).
   
   \item $a_{r-1}=b_{r-1}=N$, $a_r= b_r = E$ 
   
   
   Since $a_{r-1} = N$ by induction we have $u_{r-1} > v_{r-1}$ and
   by (L3) $u_{r-1}+u_r \leq r$ and 
   $v_{r-1} +v_r \leq r$. 
Since (C$_{r-1}$) does not hold, one of $v_{r-1} + u_r \geq r$ or 
$u_{r-1} +v_r \geq r$ must hold. By $u_{r-1} + u_r \leq r$ and
$v_{r-1} < u_{r-1}$ the inequality $v_{r-1} + u_r \geq r$
cannot hold. Hence we have $u_{r-1} +v_r \geq r$. Using
$u_{r-1}+u_r < r$ we deduce $u_r < v_r$.
\item $a_{r-1}=b_{r-1}=E$, $a_r= b_r = N$ 

The proof in this case is analogous to the preceding case.
\end{itemize}
\end{proof}

\begin{lemma} 
\label{lem:end}
  Let $P = \big(\,(a_1,u_1),\ldots, (a_d,u_d)\,\big)$ be a labeled path in $P(d,i,j+1)$ and $Q = \big(\,(b_1,v_1),\ldots, (b_d,v_d)\,\big)$ be a labeled path in $P(d,i+1,j)$. 
  Let $2 \leq s < d$ such that 
 $a_s=b_s$, $a_{s+1} = b_{s+1},
  \ldots, a_d = b_d$. 
  If either $a_{s} = b_{s} = N$ and $v_{s} < u_{s}$ or $a_{s} = b_{s} = E$ and $v_{s} > u_{s}$ holds, then there is an 
  $s \leq r < d$ such that
  one of (A$_r$), (B$_r$) or (C$_r$) holds.  
\end{lemma}
\begin{proof}
    Assume there is no $r$ such that one of (A$_r$), (B$_r$) or (C$_r$) holds. 
    By \ref{lem:parallel} it follows that 
    for $s \leq j\leq d$ we have either $a_j = b_j = N$ and $u_j > v_j$ or
    $a_j = b_j = E$ and $u_j < v_j$ . 
    
    By $P \in P(d,i,j+1)$ and $Q \in P (d,i+1,j)$ it follows that if 
    $a_d = b_d = N$ then
    $u_d=d-j-1 < d-j = v_d$ and if $a_d = b_d = E$ then 
    $u_d = j + 2 > j + 1 = v_d$. This yields a contradiction
    and hence there is an $r$ such that one of (A$_r$), (B$_r$) or (C$_r$) holds. 
\end{proof}

For the next arguments we need to consider labeled paths in $P(d)$ as actual lattice paths.
Let $P = \big(\,(a_1,u_1),\ldots, (a_d,u_d)\,\big) \in P(d,i,j+1)$ and $Q = \big(\,(b_1,v_1),\ldots, (b_d,v_d)\,\big) \in P(d,i+1,j)$.
From now interpret $N$ as a step $(0,1)$ and $E$ as a step $(1,0)$
and place with this convention the
initial point of $P$ at $(-1,1)$ 
and the initial point of $Q$ at $(0,0)$.
Using this interpretation we can consider $P$ and $Q$ as lattice paths. We speak of any
lattice point shared by the two paths as an intersection of the paths. In particular, if
we say that $P$ and $Q$ intersect after steps $(a_k,u_k)$ and $(b_k,v_k)$ we mean that
after these two steps the two points have reached the same lattice point. 
It is easy to see that $P$ and $Q$ intersect after steps $(a_k,u_k)$ and $(b_k,v_k)$ whenever
$$\# \{\, \ell~|~1 \leq \ell\leq k, a_{\ell} = N\,\}+1=
\# \{\, \ell~|~1 \leq \ell \leq k, b_{\ell} = N\,\}.$$
In particular, they intersect after the $d$\textsuperscript{th} steps 
$(a_d,u_d)$ and $(b_d,v_d)$ in the lattice point $(d-i-1,i+1)$.

Let $k \leq d$ be minimal such that $P$ and $Q$ intersect after the $k$\textsuperscript{th} step. Since $P$ starts in a lattice point, which has larger second coordinate than the
starting point of $Q$, we have $a_k = E$ and $b_k = N$. 
We first define $\Phi_k\big(\,(P,Q)\,\big) = (P',Q')$ for an incomplete set of scenarios
and later deal with the remaining situations, which we will denote by case (U). 
We distinguish $k < d$ and $k = d$. 

\noindent {\sf Case:} $k < d$ 

In the subcases of this case we swap the parts of $P$ and $Q$ that follow 
some point which $P$ and $Q$ have in common.
Since conditions (L2) and (L3) are local, in order to verify that
the resulting paths lie in $P(d)$ it suffices to check that (L2) and (L3)
hold at the common point after the swap.

\begin{itemize}
    \item $a_{k+1} = E$ and $b_{k+1} = N$ (Case (1) of \ref{squ})

    \begin{itemize}
       \item $u_k+v_{k+1} \leq k+1$ and $v_k+u_{k+1} \leq k+1$
    
    Then define $P'$ and $Q'$ by swapping the parts of $P$ and $Q$ 
    after the $k$\textsuperscript{th} step. The assumption imply that 
    (L3) is satisfied for $(P',Q')$. Moreover, 
    $(P',Q')$ is a pair of paths in 
    $P(d,i,j) \times P(d,i+1,j+1)$.
    
          \item $u_k+v_{k+1} > k+1$ or $v_k+u_{k+1} >k+1$.
          
  Here we replace the step $(a_k,u_k)
  = (E,u_k)$ by $(E,k+1-v_k)$, and replace the step $(b_k,v_k) = (N,v_k)$ by $(N,k+1-u_k)$.
  Since we have that at least one of $u_k+v_{k+1}> k+1$ and $v_k+u_ {k+1} >k+1$ holds, it follows from $u_k \geq u_{k+1}$ and $v_k\geq v_{k+1}$ that
  $u_k > k+1-v_k$ and hence $v_k > k+1-u_k$. Therefore, the replacement decreases the labels. This
  implies that (L2) and (L3) are preserved in $P$ for the steps
  $(a_{k-1},u_{k-1})$, $(a_k,u_k)$ and in $Q$ for the steps 
  $(b_{k-1},v_{k-1})$, $(b_k,v_k)$.
  
  After the replacement, swap the parts of paths of $P$ and $Q$ after the $k$\textsuperscript{th} step to obtain $(P',Q')$.
  We get one path with step $(E,k+1-v_k)$ followed by $(N,v_{k+1})$, which fulfills (L3) by $k+1-v_k+v_{k+1} \leq k+1$,
  and one path with step $(N,k+1-u_k)$ followed by $(E,u_{k+1})$, which fulfills (L3) by $k+1-u_k+u_{k+1} \leq k+1$.
  
  By \eqref{des} and \eqref{last} the pair $(P',Q')$ is a pair of 
  paths in $P(d,i,j) \times P(d,i+1,j+1)$.
  \footnote{We cannot just swap the parts of $P$ and $Q$ after the $k$\textsuperscript{th} step and get a pair of valid paths.
  Indeed, this would lead to one path with step $(E,k+1-v_k)$ followed by $(E,u_{k+1})$, and one path with step $(N,k+1-u_k)$ followed
  by $(N,v_{k+1})$. One of them would violate (L2) since $k+1-v_k<u_{k+1}$ or 
  $k+1-u_k<v_{k+1}$ in this case.}
    \end{itemize}

      \item $a_{k+1} = N$ and $b_{k+1} = E$ (Case (2) of \ref{squ})
    
    In this situation (L3) implies that
    $u_k + u_{k+1} \leq k+1$ and $v_k+ v_{k+1}\leq k+1$.
  
    \begin{itemize}
        \item $u_k \geq v_{k+1}$ and $v_k \geq u_{k+1}$

        In this case the paths $P'$ and $Q'$ obtained from $P$ and $Q$
        by swapping after the $k$\textsuperscript{th} step satisfy (L2) and hence 
        setting $\Phi_k\big(\,(P,Q)\,\big) = (P',Q')$ yields a valid pair in 
        $P(d,i,j) \times P(d,i+1,j+1)$.
        
    \end{itemize}
    
  \item $a_{k+1} = b_{k+1} = N$ (Case (3) of \ref{squ})

  In this situation (L2) and (L3) imply that
  $u_k+u_{k+1}\leq k+1$ and $v_k\geq v_{k+1}$.

  \begin{itemize}
    \item $u_k+v_{k+1}\leq k+1$ and $v_k\geq u_{k+1}$
    
    We 
    obtain $(P',Q')$ by swapping the parts of $P$ and $Q$ after the $k$\textsuperscript{th} step.
    The pair $(P',Q')$ satisfies (L2) and (L3) after the $k$\textsuperscript{th} step and 
    is a pair paths in 
    $P(d,i,j) \times P(d,i+1,j+1)$.
  
    \item $u_k+v_{k+1}> k+1$
    
    We replace the step $(a_k,u_k) = (E,u_k)$ 
    by $(E,k+1-v_k)$, and replace the step  $(b_k,v_k) = (N,b_k)$ by $(N,k+1-u_k)$.
    The assumption and $v_k \geq v_{k+1}$ imply that $k+1-v_k\leq k+1-v_{k+1} < u_k$. This
    implies that (L2) and (L3) are preserved in $P$ for the steps
    $(a_{k-1},u_{k-1})$, $(a_k,u_k)$ and in $Q$ for the steps 
    $(b_{k-1},v_{k-1})$, $(b_k,v_k)$.

    Now swap the parts of paths after the $k$\textsuperscript{th} step.
    We get one path with step $(E,k+1-v_k)$ followed by $(N,v_{k+1})$, 
    which by $k+1-v_k+v_{k+1} \leq k+1$ satisfies (L3)
    and one path with step $(N,k+1-u_k)$ followed by $(N,u_{k+1})$, which by $k+1-u_k \geq u_{k+1}$ satisfied (L2).
    By \eqref{des} and \eqref{last} the pair $(P',Q')$ is a pair of paths in $P(d,i,j) \times P(d,i+1,j+1)$.
    \footnote{We cannot swap the parts of $P$ and $Q$ after the $k$\textsuperscript{th} step and get a pair of valid paths,
    because we could get an $NN$ path with labels $k+1-u_k$ and $v_{k+1}$ which violates (L2) since $k+1-u_k<v_{k+1}$. }

  \item $v_k < u_{k+1}$ 
  
    It follows that $v_{k+1} \leq v_k < u_{k+1}$.
    Set $s = k+1$ and let $t$ be the maximal index such that
    $a_s = b_s,\ldots, a_t= b_t$.

    If there is an $s \leq r < t$ such that (A$_r$), (B$_r$) or (C$_r$)
    is satisfied then choose the smallest such $r$. 
    We then obtain $(P',Q')$ by swapping the parts of $P$ and $Q$ after the $r$\textsuperscript{th} step. The respective condition (A$_r$), (B$_r$) or (C$_r$) then
    immediately translates into (L2) and (L3) for
    $P'$ and $Q'$ after the $r$\textsuperscript{th} step.

    Now assume there is no $s \leq r < t$ 
    such that one of (A$_r$), (B$_r$) or (C$_r$) 
    holds. By $v_s < u_s$ we can apply \ref{lem:end} and deduce that
    $t < d$. 

    Again we need to distinguish cases.
    First assume that $a_t = b_t = N$ which then implies $u_t > v_t$. 

    If $a_{t+1} = E$ and $b_{t+1} = N$ then by (L2) and (L3) we have
    $u_t+ u_{t+1} \leq t+1$ and $v_t \geq v_{t+1}$. 
    Hence $v_t+u_{t+1} < u_t+u_{t+1} \leq t+1$ and 
    $u_t > v_t \geq v_{t+1}$. This implies that the paths
    $P'$ and $Q'$ obtained from $P$ and $Q$ by swapping the parts after the
    $t$\textsuperscript{th} step satisfy (L2) and (L3) after the $t$\textsuperscript{th} step. 

    The remaining case is that $a_t = b_t = E$ which then implies $u_t < v_t$.
    If $a_{t+1} = E$ and $b_{t+1} = N$ then by (L2) and (L3) we 
    have $u_t \geq u_{t+1}$ and $v_t+v_{t+1} \leq t+1$.
    Hence $v_t > u_t \geq u_{t+1}$ and $u_t+v_{t+1} < v_t + v_{t+1} 
    \leq t+1$. This implies that the paths
    $P'$ and $Q'$ obtained from $P$ and $Q$ by swapping the parts after the
    $t$\textsuperscript{th} step satisfy (L2) and (L3) after the $t$\textsuperscript{th} step.
  \end{itemize}
  \item $a_{k+1} = b_{k+1} = E$ (Case (4) of \ref{squ})
  
  In this situation (L2) and (L3) imply that
  $u_k \geq u_{k+1}$ and $v_k+ v_{k+1}\leq k+1$.
  This dual to the situation from Case (2) of \ref{squ}) and the
  same arguments and definitions apply here.

\end{itemize}

\noindent {\sf Case:} $k = d$ 

By assumption $a_d = E$ and $b_d = N$. 
Hence the $d$\textsuperscript{th} step of $P$ is
$(E,j+2)$ and the $d$\textsuperscript{th} step of $Q$ is 
$(N,d-j)$.
We set $\Phi_k(P,Q)=(P',Q')$, where $P'$ is obtained from $P$ by replacing the $d$\textsuperscript{th} step with
$(E,j+1)$ and $Q'$ is obtained from $Q$ by replacing the $d$\textsuperscript{th} step with $(N,d-j-1)$. 
The label of the $d$\textsuperscript{th} step in $P'$ is smaller than the one in $P$ and the label of the $d$\textsuperscript{th} step in 
$Q$ is smaller than the one in $Q$. It follows that (L2) and (L3) will not be violated by the
relabeling and hence by \eqref{des} and \eqref{last} $P' \in P(d,i,j)$ and $Q' \in P(d,i+1,j+1)$. 

\bigskip

Thus we have defined $\Phi_k$ except in the following
situation:

\noindent {(U)~}
         There is a $k \leq t <d$ such that $P$ and $Q$ meet after the
         $k$\textsuperscript{th} step and
         \begin{itemize}
             \item $u_k < v_{k+1}$ or $v_k < u_{k+1}$ and
             \item $a_k = E$, $b_k = N$ and
             \item $a_{k+1} = b_{k+1}, \ldots, a_t = b_t$ and 
             \item $a_{t+1} = N$, $b_{t+1} = E$.
         \end{itemize}

\begin{figure}[ht]
\centering
\begin{tikzpicture}
\draw (-1,0)--(1,0);
\draw [dashed,color=red](0,-1)--(0,1);
\draw (2,0)--(3,0)--(3,1);
\draw[dashed,color=red] (3,-1)--(3,0)--(4,0);
\draw (5,0)--(6,0)--(6,1);
\draw[dashed,color=red] (6,-1)--(6,1);
\draw (8,0)--(9,0)--(10,0);
\draw[dashed,color=red] (9,-1)--(9,0)--(10,0);
\fill (0,0) circle(0.05); \fill (3,0) circle(0.05); \fill (6,0) circle(0.05); \fill (9,0) circle(0.05);
\foreach \x in {-.5,2.5,5.5,8.5}
\foreach \y in {0}
\node at(\x,\y)[above,font=\scriptsize] {$(E,u_k)$};
\foreach \x in {0.7,9.8}
\foreach \y in {0}
\node at(\x,\y)[above,font=\scriptsize] {$(E,u_{k+1})$};
\foreach \x in {3,6}
\foreach \y in {0.7}
\node at(\x,\y)[left,,font=\scriptsize] {$(N,u_{k+1})$};
\foreach \x in {-1,2,5,8}
\foreach \y in {0}
\node at(\x,\y)[left] {P};

\foreach \x in {0,3,6,9}
\foreach \y in {-0.75}
\node at(\x,\y)[right,font=\scriptsize,color=red] {$(N,v_k)$};
\foreach \x in {0,6}
\foreach \y in {0.7}
\node at(\x,\y)[right,font=\scriptsize,color=red] {$(N,v_{k+1})$};
\foreach \x in {3.8,9.8}
\foreach \y in {0}
\node at(\x,\y)[below,font=\scriptsize,color=red] {$(E,v_{k+1})$};
\foreach \x in {0,3,6,9}
\foreach \y in {-1}
\node at(\x,\y)[below] {Q};

\foreach \x in {0,3,6,9}
\foreach \y in {0}
\node at(\x,\y)[right=0.2cm,below,font=\small] {X};

\node at(0,-2) {(1)};
\node at(3,-2) {(2)};
\node at(6,-2) {(3)};
\node at(9,-2) {(4)};

\end{tikzpicture}
\caption{Four cases in the first intersection}
\label{squ}
\end{figure}

\begin{proposition}
\label{prop:injection}
  For $d \geq 1$ and $0 \leq i,j \leq d-1$ there is an injection
  $$\Phi:P(d,i,j+1)\times P(d,i+1,j)\rightarrow P(d,i,j) \times P(d,i+1,j+1).$$
\end{proposition}
\begin{proof}
  We define the map $\Phi$ using the maps $\Phi_k$ defined before. 
  
  If $P$ and $Q$ do not intersect before reaching the endpoint then 
  $\Phi_d$ is defined and we set $\Phi((P,Q)) = \Phi_d((P,Q))$. 
  
  Now we consider the situation that $P$ and $Q$ intersect before
  reaching the end point. This implies that there is a $k < d$ such that
  $P$ and $Q$ meet after the $k$\textsuperscript{th} step and $a_k = E$, $b_k = N$.
  If $\Phi_k$ is not defined then we are in situation (U). In particular, this
  implies that if $k' > k$ is minimal such that $P$ and $Q$ intersect
  after the $k'$\textsuperscript{ th} step then $a_{k'} = E$ and $b_{k'} = N$. 
  Thus either there is a minimal $k < d$ such that $\Phi_k\big(\,(P,Q)\,\big)$ is
  defined or $\Phi_d\big(\,(P,Q)\,\big)$ is defined. We set $\Phi\big(\,(P,Q)\,\big)$ to 
  $\Phi_k\big(\,(P,Q)\,\big)$ or $\Phi_d\big(\,(P,Q)\,\big)$ respectively. 

  By the arguments define the $\Phi_k$ we now have defined a map
  $\Phi:P(d,i,j+1)\times P(d,i+1,j)\rightarrow P(d,i,j) \times P(d,i+1,j+1)$. 
  It remains to show that $\Phi$ is injective.

Let $\Phi\big(\,(P,Q)\,\big)= (P',Q')$. If $P$ and $Q$ only have the endpoint in common
then so do $P'$ and $Q'$. In particular, $P'$ and $Q'$ arose from 
relabeling the $d$\textsuperscript{th} steps which can be reversed and hence $(P,Q)$ is
determined by $(P',Q')$. If follows that $\Phi$ is injective in this
situation.

Next we consider the case that $P$ and $Q$ intersect before the endpoint. 
Assume that $\Phi\big(\,(P,Q)\,\big) = (P',Q') = \Phi\big(\,(\bar{P},\bar{Q})\,\big)$ and in both 
cases $\Phi$ is defined via a simple swap. 
If the swap points for $(P,Q)$ and $(\bar{P},\bar{Q})$ coincide then it is easy to check that
$\bar{P} = P$ and $\bar{Q} = Q$.

By symmetry we can assume that the swap point for $(P,Q)$ precedes the swap point for $(\bar {P},\bar{Q})$. Then we can write $P$ as a concatenation of $P_1P_2P_3$
and $Q$ as $Q_1Q_2Q_3$ where $P_1$, $Q_1$ are the parts from the start to
the swap point of $(P,Q)$, $P_2,Q_2$ the part from that point till the 
swap point of $(\bar{P},\bar{Q})$ and $P_3,Q_3$ the remaining paths.
We apply the same convention to $\bar{P}$ and $\bar{Q}$. 
By $\Phi((P,Q)) = (P',Q') = \Phi((\bar{P},\bar{Q}))$ it follows
that $P' = P_1Q_2Q_3 = \bar{P}_1\bar{P}_2\bar{Q}_3$ and
$Q' = Q_1P_2P_3 = \bar{Q}_1\bar{Q}_2\bar{P}_3$. It follows that,
$P_1 = \bar{P}_1$, $P_2 = \bar{Q_2}$, $P_3 = \bar{P_3}$,
$Q_1 = \bar{Q_1}$
$Q_2 = \bar{P_2}$, $Q_3 = \bar{Q_3}$.
Then $\bar{P} = P_1Q_2P_3$, $\bar{Q} = Q_1P_2Q_3$.
This shows that $\bar{P}$ and $\bar{Q}$ can swap at the same swap point 
in which $(P,Q)$ swap, but this contradicts the fact that
in the construction $\Phi$ the earliest possible swap point is used. 
It follows that $\Phi$ is injective when restricted to pairs $(P,Q)$
for which $\Phi((P,Q))$ is defined by a swap.

Now assume $\Phi\big(\,(P,Q)\,\big) = (P',Q') = \Phi\big(\,(\bar{P},\bar{Q})\,\big)$ in both cases $\Phi$ is defined by a swap after a relabeling. 
In this case an argument analogous to the the one for a simple swap
show that $(P,Q) = (P',Q')$. 

It remains to consider the case where $\Phi\big(\,(P,Q)\,\big) = (P',Q') = \Phi\big(\,(\bar{P},\bar{Q})\,\big)$ and $\Phi\big(\,(P,Q)\,\big)$ is defined by a simple swap and $\Phi\big(\,(\bar{P},\bar{Q})\,\big)$ is defined by a swap after a relabeling. 
Assume the swap happens after the $k$\textsuperscript{th} step and relabeled swap after the 
$k'$\textsuperscript{ th} step. For sure $k \neq k'$. If $k < k'$ then the labels of the 
$k'$\textsuperscript{ th} step is $P$ and $Q$ coincide with one's after relabeling 
and before swapping $\bar{P}$ and $\bar{Q}$. By the footnote on 
preceding page this cannot happen since then $\bar{P}$ and $\bar{Q}$
would not be valid paths. 
If $k' < k$ then the same argument can be applied to the $k'$\textsuperscript{th} step. 

Now we have covered all cases and it follows that $\Phi$ is injective.

\end{proof}

\begin{proof}[Proof of \ref{tp}]
  By \ref{baryh} we need to consider $H_\F = \big(\, A(d,i,j)\,\big)_{0 \leq i ,j \leq d-1}$. 
  If a minor involves the first of last row of $H_F$ then a simple
  inspection using \eqref{eq:posentry} shows that the minor is non-negative. 
  For the remaining cases, again \eqref{eq:posentry} shows that all entries in the minors are strictly positive. Hence we can apply \ref{lem:consecutive} and the non-negativity of any minor of order $2$ can be deduced from the non-negativity of all minors of $2$ consecutive rows and $2$ consecutive columns. Hence in order to show that $H_{\F}$ is TP$_2$, it suffices to show that
  $$ A(d,i,j+1)\cdot  A(d,i+1,j)\leq  A(d,i,j)\cdot  A(d,i+1,j+1),
\text{ for }  0\leq i,j\leq d-1.$$
  This inequality follows directly from \ref{prop:injection}.
\end{proof}

We close this section with another interesting property of the matrix $H_\F$
for $\F$ the barycentric subdivision.

\begin{theorem} \label{thm:inverse}
Let $\F$ be the barycentric subdivision of $(d-1)$-dimensional simplicial complexes and let 
 $P_j(x)$ be the generating polynomial of the $j$\textsuperscript{th} column of $H_{\F}^{-1}$, where $0\leq j\leq d$.
 Then  
 \begin{align}\label{gp}
    P_j(x)=\frac{1}{d!}\prod_{k=1}^{d-j-1}(-kx+k+1)\cdot\prod_{k=0}^{j-1}((k+1)x-k).
 \end{align}
\end{theorem}
\begin{proof}
By Lemma 1 and Theorem 1 from \cite{BW08} the matrix $H_\F$ can be written 
as a product
\begin{align}\label{hd}
  H_\F=A_{d+1}B_{d+1}A_{d+1}^{-1}, 
 \end{align}
 where 
 $$A_{d+1}=\left[(-1)^{i+j}\binom{d-j}{i-j}\right]_{0\leq i,j\leq d},\ \ A_{d+1}^{-1}=\left[\binom{d-j}{i-j}\right]_{0\leq i,j\leq d}$$ and 
 $$B_{d+1}=[i!S(j,i)]_{0\leq i,j\leq d},$$ 
 where $S(j,i)$ is the Stirling number of the second kind.
 Using \eqref{hd}, to show $$(1,x,\dots, x^d)H_\F^{-1}=(P_0(x),P_1(x),\dots, P_d(x)),$$
 it suffices to prove that 
 \begin{align}\label{main}
   (1,x,\dots, x^d)\,A_{d+1}= \big(\,P_0(x),P_1(x),\dots, P_d(x)\,\big)\,A_{d+1}\,B_{d+1}.
 \end{align}
 By the definition of $A_{d+1}$, it is easy to check that 
  $$(1,x,\dots, x^d)\,A_{d+1}=\big(\,(1-x)^d,x(1-x)^{d-1},\dots,x^i(1-x)^{d-i},\dots,x^d\,\big).$$
  We claim the following for the right hand side of \eqref{main}:
  
  \begin{claim}\label{claim1}
  $\big(\,P_0(x),P_1(x),\dots, P_d(x)\,\big)A_{d+1}=\big(\,f_0(x),\dots,f_d(x)\,\big),$
  where $$f_i(x)=\frac{1}{i!}(1-x)^{d-i}x(2x-1)\cdots(ix-i+1).$$
  \end{claim}
  
  \begin{proof}[Proof of \ref{claim1}]
    By the definition of $A_{d+1}$, we have 
    $f_i(x)=\displaystyle{\sum_{j=i}^{d}P_j(x)(-1)^{d-j}\binom{d-i}{j-i}}$.
    Then all we need to prove is the following identity:
    \begin{align}\label{fi}
     \sum_{j=i}^{d}P_j(x)(-1)^{d-j}\binom{d-i}{j-i}=\frac{1}{i!}(1-x)^{d-i}x(2x-1)\cdots(ix-i+1).
   \end{align}
   Take $x=\frac{a}{a+1}$ for $0\leq a\leq d-1$.
   Then 
   \begin{align*}
\text{Left hand side of } \eqref{fi} &=\sum_{j=i}^{d}P_j\big(\,\frac{a}{a+1}\,\big)(-1)^{d-j}\binom{d-i}{j-i}\\
&=\frac{1}{d!(a+1)^d}\sum_{j=0}^{a-i}(a+d-i-j)_{d}(-1)^j\binom{d-i}{j}\\
&=\frac{1}{(a+1)^d}\sum_{j=0}^{a-i}\binom{a+d-i-j}{d}(-1)^j\binom{d-i}{j}\\
&=\frac{1}{(a+1)^d}\binom{a}{i},
\end{align*}
where $(x)_{k}=x(x-1)\cdots(x-k+1)$ is the lower factorial.
The last equality holds since
$$\sum_{j=0}^{m}\binom{m+d-j}{m-j}(-1)^j\binom{d-i}{j}=\binom{m+i}{m}.$$
\begin{align*}
\text{Right hand side of } \eqref{fi} &=\frac{1}{i!}\frac{1}{(a+1)^{d-i}}\frac{(a)_i}{(a+1)^i}
=\frac{1}{(a+1)^d}\binom{a}{i}
=\text{Left hand side of } \eqref{fi}.
\end{align*}
 It follows that the polynomials on both sides of \eqref{fi} evaluate to
 the same number for  
 $d+1$ different arguments. By the fact that both polynomials have degree $d$ this completes the proof of the claim.
  \end{proof}
  
  \begin{claim}
  $\displaystyle{\sum_{i=0}^{d-1}f_i(x)i!S(j,i)=x^j(1-x)^{d-j}}$.
  \end{claim}
  \begin{proof}
  By \ref{claim1},
  \begin{align*}
  \sum_{i=0}^{d-1}f_i(x)i!S(j,i)&=
  \sum_{i=0}^{d-1}\frac{1}{i!}(1-x)^{d-i}x(2x-1)\cdots(ix-i+1)i!S(j,i)\\
 & =(1-x)^d\sum_{i=0}^{d-1}x(2x-1)\cdots(ix-i+1)(1-x)^{-i}S(j,i)\\
  &=(1-x)^d\sum_{i=0}^{d-1}(\frac{1}{1-x}-1)\cdots(\frac{1}{1-x}-i)S(j,i).
  \end{align*}
  Note that $$\sum_{k=0}^{n}S(n,k)x(x-1)\cdots(x-k+1)=x^n.$$
  We have 
  $$\sum_{i=0}^{d-1}(\frac{1}{1-x}-1)\cdots(\frac{1}{1-x}-i)S(j,i)=(\frac{1}{1-x}-1)^j,$$
  which completes the proof.
  \end{proof}
  
  Combining the two claims, we get that \eqref{main} holds.
\end{proof}

\section*{Appendix: Some well known facts about TP\texorpdfstring{$_2$}{} matrices}

For a matrix $A \in \RR^{n \times m}$, row indices $1 \leq i_1 <i_2 \leq n$ and
column indices $1 \leq j_1 < j_2 \leq m$ we denote by 
$A^{i_1,i_2}_{j_1,j_2}$ the submatrix of $A$ obtained by selecting
    rows $i_1$ and $i_2$ and columns $j_1$ and $j_2$.
    
\begin{lemma} \label{lem:consecutive}
    Let $A = (\,a_{ij}\,)_{\genfrac{}{}{0pt}{}{1 \leq i \leq n}{1 \leq j \leq m}} \in \RR^{n \times m}$ such that $a_{ij} > 0$ for $1<i<n$ and $1<j<m$.  
    Then $A$ is TP$_2$ if and only if 
    all $2 \times 2$-minors of consecutive rows and columns are non-negative.
\end{lemma}
\begin{proof}
    If $A$ is TP$_2$ then all $2 \times 2$-minors corresponding to consecutive rows and
    columns must be non-negative.
    
    Now assume that the $2 \times 2$-minors of consecutive rows and columns are non-negative.
    If there exists an entry in the first row (or column) such that  $a_{1j_1}=0$ ( $a_{i_11}=0$) then by the assumption for all $j_2>j_1$, $a_{1j_2}=0$ ($i_2>i_1$, $a_{i_21}=0$).
    
    
    Let $1 \leq i_1 < i_2 < i_2 \leq n$ and $1 \leq j_1 < j_2 < j_3\leq m$ be row and
    column indices.
    By elementary calculations we have 
    $$a_{i_2,j_2} \cdot \det(A^{i_1,i_3}_{j_1,j_2})= a_{i_3,j_2} \cdot \det(A^{i_1,i_2}_{j_1,j_2}) + a_{i_1,j_2} \cdot \det(A^{i_2,i_3}_{j_1,j_2})
    \text{ and }$$
     $$a_{i_2,j_2} \cdot \det(A^{i_1,i_2}_{j_1,j_3})= a_{i_2,j_3} \cdot \det(A^{i_1,i_2}_{j_1,j_2}) + a_{i_2,j_1} \cdot \det(A^{i_1,i_2}_{j_2,j_3}).$$
     By our assumptions $a_{i_2,j_2} > 0$ and thus 
     induction on $i_3-i_1$ and $j_3-j_1$ yields that
     $A^{i_1,i_3}_{j_1,j_2}$ and $\det(A^{i_1,i_2}_{j_1,j_3})$ 
     are non-negative. 

     It follows that all $2$-minors are non-negative.
\end{proof}

\begin{lemma} \label{lem；product}
  Let $A \in \RR^{n \times r}$ and $B \in \RR^{r \times m}$. 
  If $A$ and $B$ are TP$_2$ then so is $AB$.
  \end{lemma}
  \begin{proof}
      Let $1 \leq i_1 < i_2 \leq n$ be row indices and $1 \leq j_1 < j_2 \leq m$ be
      column indices. Then by the  Cauchy-Binet formula we obtain 
      $$\det\big(\,(AB)^{i_1,i_2}_{j_1,j_2}\,\big) = \sum_{1 \leq k_1 < k_2 \leq r}
      \det(\,A^{i_1,i_2}_{k_1,k_2}\,) \cdot \det(\,B^{k_1,k_2}_{j_1,j_2}\,).$$
      By assumption all terms on the right hand side are non-negative. It follows
      that $AB$ is TP$_2$.
  \end{proof}
  
\section*{Acknowledgments}
We thank the referees for their detailed reports and their suggestions for improvements.
The first author was partially supported by the National Natural Science Foundation of China (Grant Nos. 12271222), Scientific Research
Foundation of Jiangsu Normal University (Grant Nos. 21XFRS019) and the China Scholarship Council.

 \section*{Conflict of Interest}
 The authors declare not conflict of interest.

\section*{Data Availability}
Not applicable.

\end{document}